\documentclass[11pt,reqno]{amsart}

\usepackage{amsmath,amsthm,amsfonts}
\usepackage{latexsym,pxfonts}
\usepackage{amssymb}

\newcommand{\B}{\mathcal{B}}

\newcommand{\calE}{\mathcal{E}}

\def\f{\frac}

\newcommand{\C}{\mathbb{C}}

\newcommand{\R}{\mathbb{R}}

\newcommand{\al}{\alpha}

\newcommand{\de}{\delta}

\newcommand{\om}{\omega}
\newcommand{\la}{\lambda}

\newcommand{\s}{\sigma}

\newcommand{\La}{\Lambda}

\newcommand{\I}{\infty}

\newcommand{\EQ}[1]{\begin{equation}\begin{split} #1 \end{split}\end{equation}}
\newcommand{\Del}[1]{}

\def\calT{\mathcal{T}}

\newtheorem{thm}{Theorem}
\newtheorem{cor}[thm]{Corollary}
\newtheorem{lem}{Lemma}
\newtheorem{prop}[thm]{Proposition}

\theoremstyle{remark}

\def\diag{\mathrm{diag}}

\def\eps{\varepsilon}
\def\nn{\nonumber}

\newcommand{\calS}{\mathcal{S}}

\def\lan{\langle}
\def\ran{\rangle}

\def\tor{\mathbb{T}}
\def\calA{\mathcal{A}}
\def\de{\delta}
\def\calN{\mathcal{N}}
\def\Nbar{\overline{N}}
\def\met{\mathfrak{d}}
\def\Erw{\mathbb{E}}
\def\Prob{\mathbb{P}}

\begin{document}

\title[Regularity and convergence rates for Lyapunov Exponents]{Regularity and convergence rates for the Lyapunov exponents of linear co-cycles}

\author{W.\ Schlag}




\thanks{Support of the National Science Foundation  DMS-0617854, DMS-1160817  is gratefully acknowledged. The author thanks Michael Goldstein for
his comments on a preliminary version of this paper. } 

\begin{abstract}
We consider co-cycles $\tilde A: \tor\times K^d \in (x,v)\mapsto ( x+\om, A(x,E)v)$ with $\om$ Diophantine, $K=\R$ or $K=\C$.
We assume that $A:\tor\times \calE\to GL(d,K)$ is continuous,  depends analytically on~$x\in\tor$ and is H\"older in~$E\in \calE$, where $\calE$ is a
compact metric space. 
It is shown that if all Lyapunov
exponents are distinct at one point~$E_{0}\in\calE$, then they remain distinct near~$E$. Moreover,  they depend in a H\"older fashion on~$E\in B$
for any ball $B\subset \calE$ where they are distinct.  Similar results, with a weaker modulus of continuity, hold for higher-dimensional tori~$\tor^\nu$ with a Diophantine
shift.  We also derive optimal statements about the rate of convergence of the finite-scale Lyapunov  exponents to their infinite-scale counterparts. 
A key ingredient in our arguments is the {\em Avalanche Principle}, a deterministic statement about long finite products of invertible matrices, which goes back to work of Michael Goldstein and the author. We also discuss applications of our techniques to products of random matrices.
\end{abstract}

\maketitle

\section{Introduction}\label{sec:intro}

Let $(X,\Sigma,m;\theta)$ be an ergodic space, where $m$ is an invariant measure under the ergodic transformation~$\theta$. 
Let $A:X\to GL(d,K)$ where $K=\R$ or $K=\C$, be a measurable transformation such that
\[
\log \|A(x)\|, \; \log\|A^{-1}(x)\|\in L^1(X,m)
\]
Define the co-cycle as 
\EQ{\nn
\tilde A: X\times K^d\in (x,v)\mapsto ( \theta x, A(x)v) 
}
Powers of $\tilde A$ lead to matrix products 
$$A^{(n)}_x := \prod_{j=n}^1 A(\theta^j x).$$ 
 The classical Oseledec theorem~\cite{O}, also known as {\em multiplicative ergodic theorem},   guarantees a filtration by linear subspaces, i.e.,  for a.e.~$x\in X$ there exist linear subspaces 
\EQ{\label{fil}
K^d = V^r_x\supsetneq V^{r-1}_x \supsetneq \ldots \supsetneq V^1_x \supsetneq V^0_x=\{0\} 
}
and a sequence of positive numbers $\mu_r > \mu_{r-1}>\ldots > \mu_1$ associated with this filtration, which determine the  Lyapunov exponents  as follows.  
One has that $\dim V^j_x$ is a.e.~constant, and furthermore, the  (set) complement of $V^{j-1}_x$ in $V^j_x$ is characterized by the property that
\[
\f{1}{n} \log \|A^{(n)}_x v\| \to \log \mu_j \qquad v\in V^j_x \setminus V^{j-1}_x 
\]
as $n\to\I$. 
The {\em Lyapunov exponents} 
\EQ{\label{Lyap}
 \lambda_1\ge \lambda_2\ge \ldots \ge\la_{d-1}\ge \la_d
 }
are precisely the  numbers $\{\log \mu_j\}_{j=r}^1$ counted with multiplicity given by  $\dim V^{j}_x-\dim  V^{j-1}_x$. Finally, we remark that
 $A(x) V^j_x= V^j_{\theta x}$ for a.e.~$x$. See   also \cite{Rag, Rue, Led}.

The Lyapunov exponents are of fundamental importance in dynamics and have been widely studied, see  the encyclopedic treatment
in~\cite{BaPe}. For more recent research literature see for example~\cite{AvVi, AvVi2, BonVia, BGMV}.  Results for co-cycles with shift base (but
in a very different spirit from this paper) can be found in~\cite{A, AK, FK, K}.

\section{Statement of the main results}
\label{sec:intro2}

We mainly consider co-cycles whose base dynamics exhibits very weak mixing properties.  
Let $A:\tor\to GL(d,K)$ be continuous, and define the
co-cycle
\EQ{\label{A}
\tor\times K^d\in (x,v)\mapsto ( x+\om, A(x)v) 
}
where $\om$ is Diophantine, which will mean that  
\EQ{\label{Dio}
\| n\om\| \ge \f{c(\om)}{n(\log n)^a} \qquad\forall\; n\ge2
}
with $a>1$ arbitrary but fixed.  Lebesgue almost every~$\om\in(0,1)$ satisfy this condition. The purpose of this note is to point out a mild regularity result that the Lyapunov exponents
exhibit as functions of a parameter which $A$ is assumed to depend on. 
 To be precise, we  prove the following theorem where $(\calE,\met)$ is a compact metric space.  The study of Lyapunov exponents (especially, conditions
 which ensure positivity of the top exponent) of linear co-cycles with deterministic base dynamics goes back to the seminal work by M.\ Herman~\cite{herman}. 

\begin{thm}\label{thm:la1} 
Suppose $A:\tor\times \calE \to GL(d,K)$ is continuous ($K=\R,\C$), and analytic as a function $x\mapsto A(x,E)$ uniformly in $E\in \calE$.
Furthermore, suppose $E\mapsto A(x,E)$ is H\"older continuous, uniformly in~$x\in\tor$. Assume that the Lyapunov exponents   satisfy the {\em gap condition}
\EQ{\label{gap}
\la_j(E) -\la_{j+1}(E) >\kappa>0 \quad\forall\; E\in \calE,\, \forall \, 1\le j<d
}
Then all $\la_j(E)$ are H\"older continuous as a function of $E\in \calE$. Moreover,  
 if \eqref{gap} holds at some point $E_0\in \calE$, then each $\la_j(E)$ is H\"older continuous locally around~$E_0$. In other words, if all exponents are distinct at $E_0$,
then they are all H\"older continuous locally around~$E_0$, and therefore also remain distinct near~$E_0$. 
\end{thm}

The proof is based on an adaptation of the technique  developed jointly with Michael Goldstein in~\cite{GS} for $SL(2,\R)$ Schr\"odinger co-cycles. 
In fact, \cite{GS} dealt with Schr\"odinger co-cycles for which $E$ plays the role of the ``energy'' (spectral parameter), and one needed to impose the condition of positive
Lyapunov exponents. In Theorem~\ref{thm:la1} this positivity is exactly mirrored  by the gap condition. 
For Schr\"odinger co-cycles   H\"older regularity of the exponents is best possible; this follows from the connection between the Lyapunov exponent and
the integrated density of states as expressed by the Thouless formula, see~\cite{Bou}.

 The H\"older $C^\al$ class which our argument produces for the Lyapunov exponents depends on the size of the gaps, i.e., $\al$ depends on $\kappa$ in~\eqref{gap}.
Note that for $SL(2,\R)$ co-cycles for which the larger exponent is positive, one can refine the argument in~\cite{GS} in such a way that $\al$ does not 
depend on~$\kappa$, see \cite[Chapter 8]{Bou2}. The point here is that 
for such $SL(2,\R)$ co-cycles the exponents are $\la(E)> 0 > -\la(E)$ whence one can take $\kappa>\la$. 
This property allows one to obtain  uniform control on the H\"older exponent~$\al$
even as $\kappa\to0$.  For $GL(d,K)$ with $d>2$ there is no relation between the size of the gaps and the individual exponents, and  it 
therefore remains an open  question as to how $\al$ may 
deteriorate as $\kappa\to0$, let alone what happens when the gap condition is violated. For example, one might expect that the top exponent
remains H\"older continuous provided~\eqref{gap} holds with~$j=1$. However, the argument presented in this paper breaks down completely if, say, 
$\la_1>\la_2=\la_3$. 
Two more comments are in order: 

\begin{itemize}
\item In addition to the regularity of the exponents, we also obtain an estimate on the rate of convergence of the finite-scale
Lyapunov exponents. To be specific, we show in Section~\ref{sec:rates} that for all $n\ge1$
\EQ{\label{C/n}
|\la_{j,n}(E)-\la_j(E)|<\f{C}{n}\quad\forall\; E\in \calE
}
where $\la_{j,n}$ are the finite volume exponents. 
Moreover, one has the following dichotomy: (i) either \eqref{C/n} is optimal along a sequence of $n\to\I$ or 
(ii)  the convergence in~\eqref{C/n} is exponentially fast. 

\item It is of course natural to inquire about other types of dynamics with weak mixing properties, such as shifts $x\mapsto x+\om$ on higher-dimensional tori $\tor^\nu$, as well as shifts on~$\tor$ 
for which $\om$ satisfies weaker Diophantine conditions than~\eqref{Dio}.  In these cases one has a somewhat weaker result than Theorem~\ref{thm:la1}. While
the Lyapunov exponents are again continuous, they are not seen to be H\"older by the technology used here. Rather, one obtains regularity of the form
\EQ{\label{weaker mod}
|\la_j(E)-\la_j(E')|\le C\exp\big(-c\, |\log\met(E,E')|^\sigma\big)
}
with some $0<\sigma<1$ and $c>0$. See Section~\ref{sec:final}.
\end{itemize}

We remark that \eqref{gap} implies for all $1\le j\le d$
\EQ{
\dim V^j_x(E)-\dim V^{j-1}_x(E)= 1 \text{\ \ for a.e.\ \ }x\in\tor, \;\forall\; E\in \calE
} 
where $\{V^{j}_x(E)\}_j$ denotes the Oseledec filtration in dependence of~$E$. 

The technique developed in~\cite{GS} is based on the following two main ingredients: 
\begin{itemize}
\item The {\em Avalanche Principle} (AP) for $2\times 2$ matrices.  This is a deterministic statement. 
\item A {\em Large Deviation Theorem (LDT)} for the matrices $A^{(n)}_x$. These are quantitative versions of the  F\"urs\-tenberg-Kesten theorem, or of Kingman's
sub-additive ergodic theorem. The LDTs are analytical statements, and they depend on the structure of the co-cycles and the dynamics on the base. 
\end{itemize}

The LDT theorem from~\cite{GS}, see also~\cite{BG} and \cite{Bou2}, applies easily to higher-dimensional co-cycles. On the other hand, the
AP does not carry over from the $2\times 2$ case without changes. In fact, the gap condition is intimately tied to the formulation of the AP
as it appears in the following section.   In the final section of this paper we apply the main strategy to the products of random matrices
and derive the same type of conclusions as we did for the shift dynamics. 

This note is organized as follows.  In Section~\ref{sec:AP} we discuss the deterministic AP. In Section~\ref{sec:LDT} we discuss
the large deviation estimates for shifts on~$\tor$ under the Diophantine condition~\eqref{Dio}. Here the analytic dependence of~$A$ on $x$ is important. The proof of Theorem~\ref{thm:la1} is presented in Section~\ref{sec:proof},
and sharp convergence rates for the Lyapunov exponents are presented in Section~\ref{sec:rates}. The final two sections are very sketchy, and 
discuss other types of dynamics on the base. While Section~\ref{sec:final} deals with more general shifts, Section~\ref{sec:random} covers products
or random matrices. 

\section{The Avalanche Principle} \label{sec:AP}

The point of Lemma~\ref{lem:AP1} below, which is a completely deterministic result, is that it effectively
allows one to {\bf  linearize} long (non-commuting) matrix products $A_nA_{n-1}\ldots A_2A_1$. One cannot expect such a result for
general products, even if each factor $A_j$ has large norm. The mechanism which underlies this linearization is as follows:

\begin{itemize}

\item we assume that each matrix has a dominant {\em simple} singular value, see~\eqref{1dim}. 

\item  no two adjacent pairs $A_{j+1}A_j$ cancel, in the sense that the dominant stretching of~$A_j$ is annihilated by~$A_{j+1}$,  see~\eqref{diff}. 

\end{itemize}

Appropriate quantitative formulations of these properties guarantee that the whole product $A_n\ldots A_1$ is ``to leading order'' nothing other than
the composition of the $1$-dimensional dominant actions, see~\eqref{eq:AP}. This is what we mean by linearization. 

\begin{lem}\label{lem:AP1}
Let $\{A_j\}_{j=1}^n\subset GL(d,K)$ satisfy the following properties: for each $1\le j\le n$ there exists a $1$-dimensional
subspace $\calS_j\subset K^d$ such that 
\EQ{\label{1dim}
|A_j v| &=\|A_j\|\, |v| \quad\forall v\in \calS_j \\
|A_j w| &\le \alpha_j |w| \quad\forall w\in \calS_j^\perp \\
\|A_j\| &\ge \alpha_j \mu, \quad \mu\ge 16n^2
}
In addition, we  assume that 
\EQ{ 
\|A_{j+1}\|\,\|A_j\| < \mu^{\f14} \|A_{j+1}A_{j}\|  \qquad\forall\; 1\le j<n
\label{diff}.
}
Then
\EQ{ 
\Bigl|\log\|A_n\cdot\ldots\cdot A_1\|+\sum_{j=2}^{n-1} \log\|A_j\|-\sum_{j=1}^{n-1}\log\|A_{j+1}A_{j}\|\Bigr|
< C\frac{n}{\sqrt{\mu}}
\label{eq:AP}
}
for some absolute constant $C$. 
\end{lem}
\begin{proof}
We denote by $\pi_j$ the orthogonal projection onto $\calS_j$, and by $\pi_j^\perp$ that onto~$\calS_j^\perp$. 
Define $\calT_j:= A_j \calS_j$ and 
\EQ{\label{laj}
\la_j :=  \frac{|\pi_{j+1} v|}{|v|},\quad v\in \calT_j\setminus\{0\}
}
Since $A_j$ is invertible, we see that $\calT_j$ is a line. By a polar decomposition, each $A_j=\tilde U_j D_j U_j$ where $D_j$ is
diagonal, and $\tilde U_j, U_j$ are unitary. Moreover,   
\[
D_j = \diag(\sigma_{j,1}, \sigma_{j,2}, \ldots, \sigma_{j,d}), \quad \sigma_{j,k}>0\;\forall\, 1\le k\le d
\]
and 
\[
\max \sigma_{j,k} = \|A_j\| = \sigma_{j, k_j}, \quad \max_{k\ne k_j} \sigma_{j,k} \le \alpha_j
\]
Then 
\EQ{\nn
\| A_{j+1} A_j\| &\le \la_j \|A_{j+1}\|\, \|A_j\| + \al_j \|A_{j+1}\| + \al_{j+1} \|A_j\| \\
\| A_{j+1} A_j\| &\ge \la_j \|A_{j+1}\|\, \|A_j\| - \al_{j+1} \|A_j\|
}
whence
\EQ{\label{lajroh}
\rho_j - 2\mu^{-1}\le \la_j \le \rho_j + \mu^{-1},\qquad \rho_j := \f{ \| A_{j+1} A_j\| }{\|A_{j+1}\|\, \|A_j\|}\ge \mu^{-\f14}
}
In particular,
\[
\la_j\ge \mu^{-\f14}  - 2\mu^{-1} 
\]
Set $M_n := \prod_{j=n}^1 A_j$. Then  for any $u\in K^d$
\EQ{\label{Mn} 
M_n u = \sum_{\eps_1,\ldots, \eps_n} A_n \pi_n^{\eps_n} A_{n-1} \pi_{n-1}^{\eps_{n-1}} \cdot\ldots \pi_3^{\eps_3} A_2  \pi_2^{\eps_2} A_{1}\pi_1^{\eps_1 } u
}
where each $\eps_j$ is either the empty symbol, or equals $\perp$. 
Consider the contribution to~\eqref{Mn} where all $\eps_j$ equal the empty symbol:
\EQ{\label{eq:main}
& |A_n \pi_n  A_{n-1} \pi_{n-1}  \circ \ldots \circ\pi_3  A_2  \pi_2  A_{1}\pi_1 u| \\
& = |\pi_1 u|\prod_{j=1}^{n-1}\la_j \prod_{j=1}^n \|A_j\| 
}
Now let $I:=\{ j\in[1,n]\mid \eps_j = \perp\}$, $I^c:=[1,n]\setminus I$,  and estimate
\EQ{\label{eq:subord}
&|A_n \pi_n^{\eps_n} A_{n-1} \pi_{n-1}^{\eps_{n-1}} \circ \ldots \circ A_3 \pi_3^{\eps_3} A_2  \pi_2^{\eps_2} A_{1}\pi_1^{\eps_1 } u| \\
&\le |u| \prod_{j\in I^c}  \|A_j\| \prod_{k\in   I} \alpha_k \prod_{(j+1,j)\in I^c\times I^c} \la_j
} 
Comparing \eqref{eq:subord} with \eqref{eq:main} yields
\EQ{\nn 
& \Big( \prod_{j=1}^{n-1}\la_j \prod_{j=1}^n \|A_j\|\Big)^{-1}  \prod_{j\in I^c}  \|A_j\| \prod_{k\in   I} \alpha_k \prod_{(j+1,j)\in I^c\times I^c} \la_j \\
& =\f{\prod_{(j+1,j)\in I^c\times I^c} \la_j}{\prod_{j=1}^{n-1}\la_j} \cdot \prod_{k\in   I} \f{\alpha_k}{\|A_k\|} \\
&\le \mu^{-|I|}\Big(\prod_{j\in I\cup (I-1)} \la_j \Big)^{-1} \le \mu^{-|I|}  (\mu^{-\f14}  - 2\mu^{-1})^{-2|I|} \le 4^{|I|}  \mu^{-\f{|I|}{2}} 
}
provided $\mu^{-\f14}  \ge 4\mu^{-1}$ which is the case for $\mu\ge 8$, say.  Next, we note that
\[
\sum_{\ell=1}^n \binom{n}{\ell} 4^{\ell}  \mu^{-\f{\ell}{2}}  = (1+4\mu^{-\f12})^n - 1 \le 12 n \mu^{-\f12}
\]
provided $4n\mu^{-\f12}\le 1$, or $\mu\ge 16n^2$. Here we used that
\[
(1+x)^n -1 = \exp(n\log(1+x))-1 \le \exp(nx) -1 \le 3nx\text{\ \ if\ \ } nx\le 1
\] 
In conclusion,
\EQ{\label{oben}
\|M_n\| &\ge \prod_{j=1}^{n-1}\la_j \prod_{j=1}^n \|A_j\| \cdot \Big( 1-  \sum_{\ell=1}^n \binom{n}{\ell} 4^{\ell}  \mu^{-\f{\ell}{2}}  \Big) \\
&\ge (1-12 n \mu^{-\f12})\prod_{j=1}^{n-1}\la_j \prod_{j=1}^n \|A_j\| 
}
as well as 
\EQ{\label{unten}
\|M_n\| &\le (1+12 n \mu^{-\f12})\prod_{j=1}^{n-1}\la_j \prod_{j=1}^n \|A_j\| 
}
Next, we note that 
\EQ{\label{largo}
\prod_{j=1}^{n-1}\la_j & \le (1+\mu^{-\f34})^{n-1} \prod_{j=1}^{n-1}\rho_j   \\
\prod_{j=1}^{n-1}\la_j & \ge (1-2\mu^{-\f34})^{n-1} \prod_{j=1}^{n-1}\rho_j
}
Combining \eqref{largo} with \eqref{oben} and~\eqref{unten} implies \eqref{eq:AP}. 
\end{proof}

Note that Lemma~\ref{lem:AP1} remains unchanged if we replace any $A_j$ with a nonzero multiple of itself.
Also, for the scalar case~\eqref{eq:AP} is identically zero. 
The fact that $\calS_j$ are lines, rather than planes or higher-dimensional subspaces, 
played a crucial role in the proof. Indeed, for higher-dimensional spaces we would need to replace~\eqref{laj}
with 
\EQ{\nn
\bar \la_j :=  \sup_{v\in \calT_j\setminus\{0\}} \frac{|\pi_{j+1} v|}{|v|} 
}
Then \eqref{lajroh} remains valid with $\bar \la_j$ instead of~$\la_j$. However, 
in the estimation of the dominant branch of~\eqref{Mn} these quantities are not sufficient. 
Instead, we could only use the {\em lower bounds} 
\EQ{\nn
\underline \la_j :=  \inf_{v\in \calT_j\setminus\{0\}} \frac{|\pi_{j+1} v|}{|v|} 
}
and it is not possible to conclude as we did for $\dim(\calS_j)=1$.  Note that for the case of $SL(2,\R)$ matrices
as in~\cite{GS} the line-condition is automatic. To be precise, we simply take $\al_j=\|A_j\|^{-1}$ so that condition~\eqref{1dim}
becomes $\|A_j\|\ge 4\mu$. 

To illustrate this point further, consider orthogonal projections $P_{j}$ in~$\R^{3}$ or rank~$1$. Let
\EQ{\label{Aj1}
A_{j} = P_{j} + \eps (1-P_{j})
}
where $\eps>0$ is  small but fixed. This ensures that $A_{j}\in GL(3,\R)$. Then $\|A_{j}\|=1$, and $\|A_{j+1}A_{j}\|=|\cos \theta_{j}|$
as $\eps\to0$ where $\theta_{j}$ is the angle between the range of~$P_{j}$ and that of~$P_{j+1}$. 
Moreover, 
\[
\|A_n\cdot\ldots\cdot A_1\| \to \prod_{j=1}^{n -1} | \cos \theta_{j}| 
\]
as $\eps\to0$ whence the expression in~\eqref{eq:AP} reads
\EQ{\nn 
& \log\|A_n\cdot\ldots\cdot A_1\|+\sum_{j=2}^{n-1} \log\|A_j\|-\sum_{j=1}^{n-1}\log\|A_{j+1}A_{j}\| \\
&= \log\|A_n\cdot\ldots\cdot A_1\| - \sum_{j=1}^{n-1}\log\|A_{j+1}A_{j}\| \to 0 \text{\ \ as\ \ }\eps\to0
}
Lemma~\ref{lem:AP1} gives the precise conditions in terms of~\eqref{diff} and~\eqref{1dim} such
that~\eqref{eq:AP} holds for non-zero~$\eps$ in~\eqref{Aj1}. Note that the former condition simply reads $|\cos\theta_{j}|>\mu^{-\f14}$.

On the other hand, define  now 
\[
A_{j} := 1-P_{j} + \eps P_{j}
\]
In that case $\|A_{j+1}A_{j}\|= 1$ for any $0<\eps\le1$ since $2$-planes have a non-zero intersection in~$\R^{3}$. 
It follows that~\eqref{diff} trivially holds, as does~\eqref{1dim} at least when $\eps\to0$. 
However, \eqref{eq:AP} becomes 
\[
\|A_n\cdot\ldots\cdot A_1\|
>\exp(- C\frac{n}{\sqrt{\mu}})
\]
for any $\eps>0$ which is absurd since the left-hand side can be zero in the limit $\eps\to0$, and we may also send $\mu\to\I$
as $\eps\to0$.

\section{The Large Deviation Estimates} \label{sec:LDT}

Next, we state the LDTs for the co-cycles, see \cite[Theorem 3.8, Lemma 4.1]{GS}, and \cite[Chapter 5]{Bou2}. 
 
\begin{lem}
\label{lem:LDT}
Under the Diophantine condition on~$\omega$, see \eqref{Dio}, one has for any $\delta>0$
\EQ{
\label{LDT}
\Big |\Big \{x\in\tor \;\big|\;  | n^{-1}\log \|A_x^{(n)}(E)\| - \la_{1,n}(E)  | >  \delta \Big\} \Big| <\exp( -c\delta^2 n + C(\log n)^b)
}
where $c, C, b$ are positive constants depending on $A, \om$, and $n\ge2$.  Here
\EQ{\label{la1n}
\la_{1,n}(E) := \int_{\tor} n^{-1}\log \|A_x^{(n)}(E)\|\, dx 
}
for all $n\ge1$, $E\in \calE$. 
\end{lem}
\begin{proof}
The argument proceeds in two steps, the first step being analytical. 
The map $x\mapsto n^{-1}\log \|A_x^{(n)}(E)\|$ extends to a neighborhood $\calA$ of $\tor$ in the complex plane as a {\em subharmonic function}, uniformly in~$E\in \calE$. 
Moreover, we have
\EQ{\label{ann upper}
\sup_{z\in \calA} n^{-1}\log \|A_z^{(n)}(E)\| \le M \quad \forall\; E\in \calE
}
for some $M=M(A,\calE)$ as well as 
\EQ{\label{tor av}
&\int_{\tor} n^{-1}\log \|A_x^{(n)}(E)\| \, dx \\&\ge \int_{\tor} (dn)^{-1}\log (c(d) |\det A_x^{(n)}(E)|) \, dx \\
&= \int_{\tor} d^{-1}\log |\det A(x,E)|  \, dx  + (dn)^{-1} \log c(d) > -M \quad \forall\; E\in \calE 
}
Properties \eqref{ann upper} and \eqref{tor av} ensure that Theorem~3.8 in~\cite{GS} applies to the subharmonic function
\[
u_n(z,E) := (Mn)^{-1}\log \|A_z^{(n)}(E)\| 
\]
uniformly in $n$.  The conclusion is as follows: for all $\delta>0$, and any $m\ge1$, 
\EQ{\label{38}
&\Big |\Big \{x\in\tor \;\big|\;  \big| m^{-1}\sum_{j=1}^m u_n(x+j\om,E)  -  \lan u_n(\cdot,E)\ran  \big | >  \delta \Big\} \Big| \\ &<\exp( -c\delta m + C(\log m)^b)
}
where $\lan f\ran := \int_\tor f(x)\, dx$ is the mean, and $c,C,b$ only depend on $A,\omega$ (and not on  $m,\delta,n,E$). 

 The second step is to note the almost invariance of the co-cycles: 
\EQ{\nn 
&\sup_{x\in\tor} \Big| n^{-1}\log \|A_{x+\om}^{(n)}(E)\| - n^{-1}\log \|A_x^{(n)}(E)\| \Big| \\
& \le n^{-1}\big(\sup_{x\in\tor} \log \|A(x,E)\|+ \sup_{x\in\tor}\log \|A^{-1}(x,E)\| \big) = \f{C}{n}
}
Upon iteration, we have 
\EQ{\label{almost inv}
\sup_{x\in\tor} |u_n(x+k\om,E) -  u_n(x,E)| \le  \f{Ck}{n}
}
To obtain \eqref{LDT} we now set $m:= \delta n/C$ in~\eqref{38} and conclude via~\eqref{almost inv}. 
\end{proof} 

By general principles,
\[
\la_{1,n}(E) \to \la_1(E) \text{\ \ as\ \ }n\to\I
\]
as a monotone decreasing sequence.  Indeed, it follows from the co-cycle relation
\[
A^{(n+m)}_x = A^{(n)}_{x+m\om}\circ A^{(m)}_x
\]
that
\[
\log\|A^{(n+m)}_x\|\le \log\|A^{(n)}_{x+m\om}\|+\log\| A^{(m)}_x\|
\]
Averaging in~$x$ implies $(n+m)\la_{1,n+m} \le n \la_{1,n} + m\la_{1,m}$, whence the claim. 

We now turn to the exterior powers of $A^{(n)}_x$, which we denote by $\Lambda^p A^{(n)}_x$. 
On a general vector space~$V$ with a linear operator $L:V\to V$ one defines 
\[
(\Lambda^p L) f (\xi_1,\ldots,\xi_p) := f(L^* \xi_1,\ldots, L^* \xi_p) \qquad\forall\; \xi_j\in V^*
\]
where $f$ is an alternating linear form on $(V^*)^p$. 
We have the following more general version of the LDT,
which  allows us to control the Lyapunov exponents $\la_j$ with $j\ge2$. 

\begin{lem}
\label{lem:LDT2}
Under a suitable Diophantine condition on~$\omega$ one has for any $\delta>0$
\EQ{
\nn 
\Big |\Big \{x\in\tor \;\big|\;  \big| n^{-1}\log \| \Lambda^p A_x^{(n)}(E)\| - \sum_{i=1}^p \la_{i,n}(E)  \big| >  \delta \Big\} \Big| <\exp( -c\delta^2 n + C(\log n)^b)
}
where $c, C, b$ are positive constants depending on $A, \om$, and $n\ge2$. The $\lambda_{i,n}$ are defined inductively via
\[
\int_{\tor} \f{1}{n}\log \| \Lambda^p A_x^{(n)}(E)\| \, dx =  \sum_{i=1}^p \la_{i,n}(E) 
\]
with $1\le p\le d$. 
\end{lem}
\begin{proof}
The proof is essentially the same as that of the previous lemma; indeed, we retain the analyticity of 
\[
z\mapsto \Lambda^p A_z^{(n)}(E)
\]
on the annulus $\calA$ uniformly in $E\in \calE$. This allows us to apply the subharmonic machinery as before. Moreover, we have the product structure
\[
\Lambda^p A_x^{(n)}(E) = \prod_{j=n}^1 \Lambda^p A(x+j\om,E)
\]
which ensures the almost invariance \eqref{almost inv} (note that $\Lambda^p A$ remains invertible).  So the argument leading to the previous lemma 
applies. 
\end{proof}

Write (dropping $E$ for simplicity) $A_x^{(n)} = \tilde U_x^{(n)}\circ D_x^{(n)} \circ U_x^{(n)}$, where $\tilde U_x^{(n)}, U_x^{(n)}$ are unitary, and $D_x^{(n)}$
is diagonal with entries
\[
\s_{1,n}(x)\ge \s_{2,n}(x) \ge \ldots \ge \s_{d-1,n}(x)\ge \s_{d,n}(x)>0
\]
It follows that $\Lambda^p A_x^{(n)} = \Lambda^p\tilde U_x^{(n)}\circ \Lambda^p D_x^{(n)} \circ \Lambda^p U_x^{(n)}$, whence
\EQ{\nn
\int_{\tor} \f{1}{n}\log \| \Lambda^p A_x^{(n)} \| \, dx &= \int_{\tor} \f{1}{n}\log \| \Lambda^p D_x^{(n)} \| \, dx \\
&= \sum_{i=1}^p \f{1}{n}\int_{\tor} \log \s_{i,n}(x)\, dx = \sum_{i=1}^p \la_{i,n}
}
since $\Lambda^p U$ is an isometry if $U$ is. Consequently, 
\EQ{\label{sigla}
\f{1}{n} \int_{\tor}\log \s_{i,n}(x)\, dx = \la_{i,n} \quad\forall\; 1\le i\le d
}
In addition, we note that as decreasing limits of continuous functions, the sums
\[
\sum_{i=1}^p \la_{i}(E),\quad 1\le p\le d
\]
are upper-semicontinuous. But more regularity is harder to obtain. 

\section{The proof of Theorem~\ref{thm:la1}}
\label{sec:proof}
 
\begin{proof}[Proof of Theorem~\ref{thm:la1}] 
Normalizing, we may assume that $A:\tor\to SL(d,K)$. This implies
\EQ{
\label{zero sum}
\sum_{j=1}^d \la_j(E)=0
}
For simplicity, we consider the case of the
two largest Lyapunov exponents first, and our goal is to show that the largest Lyapunov exponent is H\"older. 
In other words, $j=1$ in~\eqref{gap}. 
Fix $E_0\in \calE$. We may choose $n_0$ such that 
\EQ{
|\la_{i,n}(E_0)-\la_i(E_0)| < \de_0 \quad \forall\; n\ge n_0, \; 1\le i\le d
}
where $\de_0>0$ is a small constant which will be specified later. By \eqref{sigla}  
\EQ{
\f{1}{n} \int_{\tor}\log \| A^{(n)}(x,E_0)\|\, dx&= \f{1}{n} \int_{\tor}\log \s_{1,n}(x,E_0)\, dx = \la_{1,n}(E_0)\\
&\ge \la_1(E_0)  \ge \kappa + \la_2(E_0) \\
&\ge \kappa  -\de_0 + \la_{2,n}(E_0) \\&= \kappa  -\de_0 + \f{1}{n} \int_{\tor}\log \s_{2,n}(x,E_0)\, dx
}
Define $\B_n\subset \tor$ as the set of $x\in\tor$ for which 
\EQ{
\Big| n^{-1}\log \| \Lambda^p A_x^{(n)}(E_0)\| - \sum_{i=1}^p \la_{i,n}(E_0) \Big|  > \de_0 
}
with $p=1$ or $p=2$. 
By Lemma~\ref{lem:LDT2}, $|\B_n|< e^{-c\de_0^2 n}$ for $n$ large enough.  On $\B_n^c$ we have
\EQ{\nn 
\f{1}{n}  \log \s_{1,n}(x,E_0) &> \la_{1,n}(E_0)-\de_0 \ge \la_{1}(E_0)-\de_0 \\
\f{1}{n}  \log \s_{2,n}(x,E_0) &< \la_{2,n}(E_0)+2\de_0 \le \la_{2}(E_0)+3\de_0
}
whence
\EQ{\label{srat}
\s_{1,n}(x,E_0) \s_{k,n}(x,E_0)^{-1} \ge e^{n(\kappa-4\de_0)} \quad \forall \;2\le k\le d
}
Let $N>n$ be a multiple of $n$ (for simplicity of notation) and write
\[
A_x^{(N)}(E_0) = \prod_{j=k-1}^0 A_{x+jn\om}^{(n)}(E_0)
\]
We wish to apply the AP to this product with $A_j:=A_{x+jn\om}^{(n)}(E_0)$. In order to do this, we take $N= \lfloor e^{\de_1 n}\rfloor$ where $0<\de_1\ll \de_0^2$ is
another constant. The union of $N/n=k$ shifted copies of $\B_n$ has measure
\[
< \f{N}{n} e^{-c\de_0^2 n} < e^{-(c\de_0^2-\de_1) n} < e^{-c\de_0^2 n/2}
\]
Denoting this set again by $\B_n$, we see that conditions~\eqref{1dim} hold with $\mu$ determined by~\eqref{srat}, i.e., 
\[
\mu = e^{n(\kappa-4\de_0)} > N^{C}
\]
where $C$ is a large constant, say $C=\kappa/(2\de_1)$. By the LDT applied to $A^{(2n)}_x(E_0)$ and $A^{(n)}_x(E_0)$, we may also
assume that~\eqref{diff} holds on~$\B_n^c$ for all $0\le j<k$. 

Hence, \eqref{eq:AP} holds, viz.,
\EQ{\label{ANn}
\Bigl|\log\| A_x^{(N)}(E_0) \|+\sum_{j=1}^{k-2} \log\|A_{x+jn\om}^{(n)}(E_0) \|-\sum_{j=0}^{k-2}\log\|A_{x+jn\om}^{(2n)}(E_0)\|\Bigr|
< C\frac{N}{n\sqrt{\mu}} 
}
provided $x\in \B_n^c$. The right-hand side of~\eqref{ANn} is of the form $N^{-C}$ for a large constant~$C$, as is the measure of~$\B_n$. Averaging~\eqref{ANn}
over $\B_n^c$ therefore yields 
\EQ{
\nn
| N\la_{1,N}(E_0) + (k-2) n \la_{1,n}(E_0)  - 2(k-1)n \la_{1,2n}(E_0) | < N^{-C}
}
which further implies, for all $n\ge n_0$, and with $N= \lfloor e^{\de_1 n}\rfloor$ where $\de_1$ depends on~$\kappa$, 
\EQ{
\label{GS}
| \la_{1,N}(E_0) +   \la_{1,n}(E_0)  - 2  \la_{1,2n}(E_0) | < C\f{n}{N}= C(\kappa) \f{\log N}{N}
}
We now argue that~\eqref{GS} remains valid if we replace $E_0$ with a nearby~$E$. To see this, we first run the argument leading up to~\eqref{GS} for all $n\in [n_0, e^{n_0}]=:\calN_0$
and note that~\eqref{srat} in that case remains valid near~$E_0$; in fact on some ball $B:=B(E_0,\eps)\subset\calE$ 
where $\eps$ depends on~$\kappa$ and~$n_0$. For example, we may take $\eps=\exp(-Ce^{n_0})$ with some large~$C$. 
Furthermore, the condition~\eqref{1dim} and~\eqref{diff} also remain valid on that ball. Therefore, \eqref{GS} does hold
for $E\in B$  with $n\in\calN_0$. In particular, with $$N\in \{\lfloor e^{\de_1 n}\rfloor\mid n\in \calN_0\}=:\calN_1,$$ we have 
\EQ{\label{gapu}
\la_{1,N}(E) &>  2  \la_{1,2n}(E)-\la_{1,n}(E) -C(\kappa) \f{\log N}{N} \\
&> \la_1(E_0)-\de_0 -C(\kappa) \f{\log N}{N}  \qquad \forall\; E\in B
}
The idea is now to combine this estimate with one for $\la_{2,N}(E)$. This will allow us to show that the gap between the $\{\la_{j,n}(E)\}_{j=1}^d$ does not shrink by much when we pass from scale $n\in\calN_0$ to 
 scale~$N\in\calN_1$ with $E\in B$. If $d=2$ we are of course done, see~\eqref{zero sum}. If $d>2$, then we need to re-run the AP-argument for $\Lambda^2 A^{(n)}_x(E_0)$.  
This is legitimate, since the two largest exponents of the $\Lambda^2 A$ co-cycle are $$\la_1(E_0)+\la_2(E_0) > \kappa+ \la_1(E_0)+\la_3(E_0).$$
The same AP argument as before therefore yields
\EQ{
\label{GS2}
| \la_{2,N}(E_0) +   \la_{2,n}(E_0)  - 2  \la_{2,2n}(E_0) | < C\f{n}{N}= C(\kappa) \f{\log N}{N}
}
and thus, with $N\in\calN_1$ related to~$n\in\calN_0$ as before, 
\EQ{\nn 
\la_{2,N}(E) &<  2  \la_{2,2n}(E)-\la_{2,n}(E) + C(\kappa) \f{\log N}{N} \\
& < \la_2(E_0)+2\de_0 + C(\kappa) \f{\log N}{N}  \qquad \forall\; E\in B
}
Subtracting this from~\eqref{gapu} yields
\EQ{\label{gap2}
\la_{1,N}(E) - \la_{2,N}(E) > \la_1(E_0)-\la_2(E_0)-3\de_0 -C(\kappa) \f{\log N}{N}  \qquad \forall\; E\in B
}
By the same token, we have this property for all gaps, i.e., for any $1\le j<d$, 
\EQ{\nn
\la_{j,N}(E) - \la_{j+1,N}(E) > \la_j(E_0)-\la_{j+1}(E_0)-3\de_0 -C(\kappa) \f{\log N}{N} >\kappa\qquad \forall\; E\in B
}
provided $\de_0$ is sufficiently small, and $n_0$ large. 

We claim that we may now iterate this argument without shrinking~$B$. First, we define 
$$  \{\lfloor e^{\de_1 n}\rfloor\mid n\in \calN_\ell\}=:\calN_{\ell+1}$$
for each $\ell\ge1$ and note that by construction $$\cup_{\ell\ge0}\calN_\ell = [n_0,\I)$$
In fact, $\calN_{\ell}$ and $\calN_{\ell+1}$ overlap.  To $n\in\calN_0$ associate $N\in\calN_1$
as above, and then set $\Nbar = \lfloor e^{\de_1 N}\rfloor\in\calN_2$. 
We may now repeat the argument leading to~\eqref{GS}, \eqref{GS2}, respectively, 
for the $N,\Nbar$ scales.  Indeed, 
 these estimates now  become
\EQ{
\label{GS2*}
| \la_{1,\Nbar}(E) +   \la_{1,N}(E)  - 2  \la_{1,2N}(E) | & < C(\kappa) \f{\log \Nbar}{\Nbar} \\
| \la_{2,\Nbar}(E) +   \la_{2,N}(E)  - 2  \la_{2,2N}(E) |  & < C(\kappa) \f{\log \Nbar}{\Nbar}
}
for all $E\in B$. 
Next, we  replace $\la_{1,N}(E_0)$ in~\eqref{GS} (with $E\in B$ instead of~$E_0$)
with $\la_{1,2N}(E)$. Subtracting the resulting estimate from~\eqref{GS} implies that  
\EQ{\label{N2N}
|\la_{1,N}(E)-\la_{1,2N}(E)| < C(\kappa) \f{\log N}{N}\qquad \forall\; E\in B
}
for all~$N\in\calN_1$,   and similarly for the other exponents.  In combination with~\eqref{GS2*} and~\eqref{gap2}
this allows us to conclude that for all $E\in B$, 
\EQ{ \label{Nbar}
\la_{1,\Nbar}(E)-\la_{2,\Nbar}(E) &> \la_{1,N}(E)-\la_{2,N}(E) - C(\kappa) \big( \f{\log N}{N} + \f{\log \Nbar}{\Nbar} \big)\\
&> \la_1(E_0)-\la_2(E_0)-3\de_0 - C(\kappa) \big( \f{\log N}{N} + \f{\log \Nbar}{\Nbar} \big)
}
It is essential here that we do not lose more factors of~$\de_0$, but only subtract terms such as on the right-hand side
which are summable over a sequence of scale $N_j\in \calN_j$.  In view of~\eqref{Nbar} we can indeed repeat the arguments
again for the next scale $N':= \lfloor e^{\de_1 \Nbar}\rfloor\in\calN_3$ and so on. This establishes our claim concerning 
infinite iterations while keeping the ball~$B$ fixed.

In particular, \eqref{N2N} will hold for all large $N$.
Summing these estimates over a sequence~$2^kN$ implies that 
\EQ{\label{crude rate}
|\la_{j,N}(E)-\la_{j}(E)| < C(\kappa) \f{\log N}{N}\qquad \forall\; E\in B 
}
and all large~$N$ and all $1\le j\le d$.  In particular, we have the uniform gaps 
\EQ{\label{open}
\la_{j}(E) - \la_{j+1}(E) > \kappa \qquad \forall\; E\in B 
}
Note that we have shown that the validity of~\eqref{open} at one point~$E_0$ implies its validity for all~$E$ near~$E_0$.

To  establish the H\"older regularity of the exponents, we 
first replace~$\la_{1,N}(E_0)$ with $\la_{1}(E)$ in~\eqref{GS} which yields
\EQ{
\label{GS3}
| \la_{1}(E) +   \la_{1,n}(E)  - 2  \la_{1,2n}(E) | < C\f{n}{N}= C(\kappa) \f{\log N}{N} \qquad \forall\; E\in B
}
and all large $N$ (the relation between $n$ and~$N$ is as before). The idea is now to pick two $E,E'\in B$
and then to compare the resulting estimates~\eqref{GS3} with an $N$ that is adjusted to $\met(E,E')$, the distance between $E,E'\in \calE$. The be specific, 
we claim the crude bound
\EQ{\label{crude}
|\la_{1,n}(E) - \la_{1,n}(E')|\le e^{C_0 n}\,\met(E,E')^{\beta_0},
}
for some $\beta_0\in(0,1]$. This follows from the assumed H\"older regularity of $E\mapsto A(x,E)$ with a 
a constant $C_0$ that depends on $A$. Indeed,  expanding the co-cycle product and writing the difference in a telescoping fashion
leads to the estimate 
\EQ{\label{crude2}
\| A^{(n)}_x(E)-A^{(n)}_x(E')\|\le e^{Cn}\, \met(E,E')^{\beta_0}
}
Furthermore, $$\|A^{(n)}_x(E)\|\ge c(d) |\det A^{(n)}_x(E)|^{\f{1}{d}} \ge  ce^{-Cn}  $$ whence 
\EQ{ \nn 
|\la_{1,n}(E) - \la_{1,n}(E')| &\le \f{1}{n}\int_\tor \log\Big( 1+ \f{ \| A^{(n)}_x(E)-A^{(n)}_x(E')\| }{ \| A^{(n)}_x(E)\| }\Big)\, dx \\
&\le \sup_{x\in\tor}  \f{ \| A^{(n)}_x(E)-A^{(n)}_x(E')\| }{ \| A^{(n)}_x(E)\| }  \le Ce^{Cn}\, \met(E,E')^{\beta_0}
}
 for all $n\ge1$ as claimed. 
 
 We infer from~\eqref{GS3}, \eqref{crude} that
\EQ{
\label{la1Ediff}
| \la_{1}(E) - \la_1(E')| &\le e^{C_0 n}\,\met(E,E')^{\beta_0} +    C(\kappa) \f{\log N}{N} \\
& \le \met(E,E')^{\beta_0}\, N^{C_0\de_1^{-1}}  +    C(\kappa) \f{\log N}{N}  \qquad \forall\; E,E'\in B
}
Optimizing over~$N$ yields the desired result:
\EQ{
\label{Holder}
| \la_{1}(E) - \la_1(E')| &\le  C(\kappa) \met(E,E')^{\gamma}  \qquad \forall\; E,E'\in B
}
where $\gamma=\gamma(\kappa,\beta_0)$. The same argument applies to the other exponents by repeating these considerations for
the powers $\Lambda^p A^{(n)}_x(E)$.  
\end{proof}

It is natural to ask whether $\la_{1}(E)$ remains H\"older if $\la_1(E)-\la_2(E)>0$. In other words, if there is a gap between
the two largest Lyapunov exponents, does it follow that the top exponent is H\"older? The proof we just gave does not show
this, since it relies on statements about $\la_{2,n}(E)$ over different scales~$n$, uniformly in~$E$. Such control can only be obtained
via the AP, at least within the confines of our methods. However, the AP for $\La^2 A^{(n)}_x$  requires a gap between $\la_2$ and~$\la_3$,
and so on. So it is really important for our argument that all exponents are distinct, even to show that the top one is H\"older regular. 

We remark that the proof we just gave is not optimal in some ways. The main improvements one can make relate to various
upper bounds we used, such as~\eqref{crude} where one can insert the Lyapunov exponent~$\la_{1}(E)$ into the exponent, viz. 
\EQ{\label{less crude}
|\la_{1,n}(E) - \la_{1,n}(E')|\le e^{2 \la_{1}(E) n}\,\met(E,E')^{\beta_0},
}
for large $n$; here we are using that $\la_1(E)>0$ which follows from~\eqref{gap}. 
To obtain~\eqref{less crude}, one relies on the upper bound
\[
\sup_{x\in\tor}  \log\|A^{(n)}_x(E)\| \le n\la_{1}(E) + (\log n)^{C_0} 
\]
for large $n$, and analogously
\[
\sup_{x\in\tor} \log\|\Lambda^p A^{(n)}_x(E)\| \le n\sum_{j=1}^p \la_{j}(E) + (\log n)^{C_0} 
\]
provided that $\sum_{j=1}^p \la_{j}(E)>0$. 
Here $C_0$ is a constant that depends on~$\om$, and~$A$.  For a proof, see~\cite[Proposition 4.3]{GS2}. Results of this
type played an important role in the finer results required for the Cantor structure of the spectrum, see~\cite{GS2, GS3}. 

Moreover, they were used by Bourgain in~\cite[Chapter 8]{Bou2} to show that the H\"older exponent does not depend on the Lyapunov exponent (as
long as it is positive) for Schr\"odinger co-cycles.  In addition to these finer upper bounds one also needs to improve on~\eqref{LDT} to 
accomplish this, namely by removing the dependence on the~$\de^2$ on the right-hand side. The appearance of this $\de^2$, too, is closely related to a crude upper bound, see~\eqref{almost inv}. Indeed, it is natural that we may place the Lyapunov exponent on the right-hand side of~\eqref{almost inv}, which then shows
that for $\de\simeq \la_1$ we may put~$\de$ on the right-hand side of~\eqref{LDT} rather than~$\de^2$. This is exactly what is done in~\cite[Chapter 8]{Bou2}. 

However, we do not seem to gain anything from these improvements which is why we have not implemented them in the general $GL(d,\R)$ case. As already
mentioned in Section~\ref{sec:intro2}, for $d>2$ there is no relation between the size of the gaps and the Lyapunov exponents themselves. Thus, 
unlike the $d=2$ case the methods of this paper cannot possibly lead to gap-independent H\"older classes.

\section{Rates of convergence}\label{sec:rates}

As a byproduct of the argument presented in the previous section we obtained the convergence rates~\eqref{crude rate}, viz. 
\EQ{\nn
|\la_{j,n}(E)-\la_{j}(E)| < C(\kappa) \f{\log n}{n}\qquad \forall\; E\in \calE,\; 1\le j\le d
}
for all $n\ge2$. Here $\calE$ is as in Theorem~\ref{thm:la1} (cover the compact space $\calE$ by finitely many open balls~$B$ as in the previous section). 
The purpose of this section is to remove the $\log n$ factor from this rate. This is the analogue of~\cite[Theorem 5.1]{GS} where
the same result was obtained for $SL(2,\R)$ Schr\"odinger co-cycles. 

\begin{prop}
\label{prop:rates} Under the assumptions of Theorem~\ref{thm:la1} one has 
\EQ{\label{good rates}
|\la_{j,n}(E)-\la_{j}(E)| <  \f{C}{n}\qquad \forall\; E\in \calE,\; 1\le j\le d
}
for all $n\ge1$. The constant $C$ depends on $\kappa, A, \calE,\omega$. 
\end{prop}
\begin{proof}
We write,
\[
A^{(2n)}_x(E) = A^{(n)}_{x+n\om}(E) \circ A^{(n)}_x(E)
\]
The idea is now to apply the AP to the three matrix products $A^{(2n)}_x(E)$, $A^{(n)}_{x+n\om}(E)$, and $A^{(n)}_x(E)$
with factors of the form $ A^{(\ell_j)}_{x+m_j\om}(E)$. Here $n$ is large and 
$$
C_1\log n \le \ell_j \le 2C_1\log n, \quad m_j = \sum_{i=1}^{j-1} \ell_i
$$
for some large $C_1$. This can be justified as in Section~\ref{sec:proof} for all  $x\in \B_n^c$ where $|\B_n|< n^{-10}$, say (taking $C_1$ large). 
Subtracting the resulting representation~\eqref{eq:AP} for $\log\|A^{(n)}_{x+n\om}(E)\|$ and $\log \|A^{(n)}_x(E)\|$ from that for $\log \|A^{(2n)}_x(E)\|$
yields the estimate
\EQ{
\label{eq:2scales}
& \big| \log \|A^{(2n)}_x(E)\| - \log\|A^{(n)}_{x+n\om}(E)\| - \log \|A^{(n)}_x(E)\| \\
&- \log \|A^{(2\ell)}_{x+(n-\ell)\om}(E)\| + \log\|A^{(\ell)}_{x+n\om}(E)\| + \log \|A^{(\ell)}_{x+(n-\ell)\om}(E)\| \; \big| < n^{-1}
}
for all $x\in \B_n^c$ and some $\ell\simeq \log n$. The power $n^{-1}$ can be improved on the right-hand side but this is of no consequence. 
Integrating~\eqref{eq:2scales} over $\tor$ and noting that the integral over $\B_n$ makes a negligible contribution yields (dropping $E$ for simplicity)
\EQ{\nn
& \big| 2n (\la_{1,2n} - \la_{1,n})    -   2\ell(\la_{1,2\ell}-\la_{1,\ell})   \big| < Cn^{-1}
}
Setting $R(n):= 2n |\la_{1,2n} - \la_{1,n}| $ we infer that
\[
R(n) \le R(\ell) + \frac{C}{n}
\]
Iterating and summing yields $R(n)\le C$ for all $n\ge1$ as claimed.  Applying the exact same reasoning to the exterior
powers $\Lambda^p A$ yields the analogous estimate for the other exponents. 
\end{proof}

We now turn to the question of whether \eqref{good rates} is an optimal estimate. As the example of constant $A$ shows,
this of course need not be the case. However, based on the observation from Section~\ref{sec:proof} that 
 the  quantities $\la_{j,2\ell}(E)-\la_j(E)$ and $\la_{j,\ell}(E)-\la_{j,2\ell}(E)$ differ only by an amount that is {\em exponentially small in~$\ell$},
 we now establish the following dichotomy: either \eqref{good rates} is optimal, or the convergence is {\em exponentially fast}. 

\begin{cor}\label{cor:geometric} 
Under the assumptions of Theorem~\ref{thm:la1} there exists a constant $c_1=c_1(\kappa,A,\calE,\omega)>0$ such that 
\EQ{\label{exp} |\la_{j}(E)- 2\la_{j,2\ell}(E)+\la_{j,\ell}(E)|\le \exp(-c_1\ell) 
} 
for all sufficiently large~$\ell$ and all $1\le j\le d$. 
Moreover, there is $\ell_0=\ell_0(c_1)$ so that if $$|\la_{j,\ell_1}(E)-\la_j(E)|>4\exp(-c_1\ell_1)$$ for some $\ell_1\ge\ell_0$ and some $j$, then
\[
|\la_{j,2^k\ell_1}(E)-\la_j(E)|>\frac{1}{2^{k+1}}|\la_{j,\ell_1}(E)-\la_{j}(E)|
\] for all $k\ge 0$. In other words, on balls $B$ on which~\eqref{gap} holds, we have the following property: for each $E\in B$,  either $\la_{j,n}(E)\to \la_j(E)$ exponentially fast, or 
\EQ{\label{1overn}
|\la_{j,n}(E)-\la_j(E)|>\frac{C(E)}{n}
}
 for infinitely many~$n$. 
\end{cor}
\begin{proof} 
The relation \eqref{exp} is a restatement of \eqref{GS3}. 

Now assume that $\la_{j,\ell_1}(E)-\la_j(E)>4\exp(-c_1\ell_1)$ where $j$ is fixed. In view of \eqref{exp},
\[ |\la_{j,2\ell_1}(E)-\la_j(E)| > \f12 |\la_{j,\ell_1}(E)-\la_j(E)|-\f 12 \exp(-c_1\ell_1).\]
Continuing inductively one obtains that
\EQ{\label{bremse} |\la_{j,2^k\ell_1}(E)-\la_j(E)| &> \frac{1}{2^k}|\la_{j,\ell_1}(E)-\la_j(E)| \\
& -\frac{1}{2^k}\exp(-c_1\ell_1)\Bigl[1+2\exp\bigl(-(2-1)c_1\ell_1\bigr) +\ldots \\
& + 2^{k-1}\exp\bigl(-(2^{k-1}-1)c_1\ell_1\bigr)\Bigr].
}
Now choose $\ell_0$ so large that
\[ \sum_{j=0}^\infty 2^j\exp\bigl(-(2^j-1)c_1\ell_0\bigr)\le 2.\]
By \eqref{bremse} and our assumption,
\EQ{\nn
| \la_{j,2^k\ell_1}(E)-L(E)| &> \frac{1}{2^k}|\la_{j,\ell_1}(E)-\la_j(E)|-\frac{2}{2^k}\exp(-c_1\ell_1)\\ &>\frac{1}{2^{k+1}}|\la_{j,\ell_1}(E)-\la_j(E)|
 }
as claimed. 
\end{proof}

\section{Other types of shifts}
\label{sec:final}

As already noted in Section~\ref{sec:intro2}, the proof is modular and rests on two main ingredients, the avalanche principle
on the one hand, and the large deviation theorems on the other hand. In fact,  the proof in Section~\ref{sec:proof} applies to 
any $GL(d,K)$ cocycle over an "abstract" base dynamics as in Section~\ref{sec:intro} $\theta:X\to X$ provided one has an LDT
\EQ{\label{XLDT}
m\big(\big \{x\in X \;\big|\;  \big | n^{-1}\log \|\La^p A_x^{(n)}(E)\| - \sum_{i=1}^p \la_{i,n}(E) \big | >  \delta \big\}\big)  <\exp( -c(\de) n )
}
for all $\de>0$ and all sufficiently large $n\ge n_0(\de)$ with some constant $c(\de)>0$.   

In fact, a weaker statement such as 
\EQ{\label{weakLDT}
m\big( \big \{x\in X \;\big|\;  \big | n^{-1}\log \|\La^p A_x^{(n)}(E)\| - \sum_{i=1}^p \la_{i,n}(E) \big | >  \delta \big\} \big) <\exp( - n^\tau )
}
for all $\de>0$, $n\ge n_0(\de)$ and some $\tau=\tau(\delta)>0$ would still lead to a result. But instead of H\"older regularity one
obtains a modulus of the type~\eqref{weaker mod}.  For example, instead of~\eqref{exp} we only obtain
\[
|\la_{j}(E)- 2\la_{j,2\ell}(E)+\la_{j,\ell}(E)|\le \exp(-\ell^\sigma),\qquad 0<\sigma<1
\]
under the assumption~\eqref{weakLDT}. In the iteration which underlies the proof of Theorem~\ref{thm:la1} given in Section~\ref{sec:proof} we can therefore
only pass from scale $n$ to $N=\exp(n^\beta)$ for some $0<\beta<1$.  The details are routine, and are left to the reader. 

Examples of dynamics that give rise to LDTs of the form~\eqref{weakLDT} are shifts on higher tori~$\tor^\nu$ as well as those on~$\tor$
for which~\eqref{Dio} is relaxed such as to
\[
\| n\om\| \ge c(\om)\, n^{-a} \qquad\forall\; n\ge1
\]
where $a>1$ is arbitrary but fixed. See for example~\cite[Proposition 10.2]{GS}, \cite[Proposition 7.18]{Bou2} for  precise 
statements along these lines.  It is not known if these weaker results can be improved or not; in other words, for higher-dimensional tori
we currently do not have an estimate such as~\eqref{XLDT}. 

\section{Products of random matrices}
\label{sec:random}
 
For the sake of completeness, and less to say anything new, we now sketch how the machinery of this note relates to the classical theory
of  products of random matrices. A standard reference for everything that we will need is the book by Bougerol, Lacroix~\cite{BL}; see
also Le Page~\cite{Lepage}, and Ledrappier's lecture notes~\cite{Led}. 

Let $\mu$ be a probability measure on~$GL(d,\R)$ and consider i.i.d.~variables $\{Y_j\}_{j=1}^\I$ with common distribution~$\mu$ such that $\Erw[\log^+ \|Y_1\|]<\I$. We assume that
$\mu$ is {\em strongly irreducible}, i.e., there is no finite union of proper subspaces of~$\R^d$ which is invariant under every matrix in~$T_\mu$, the 
smallest closed semigroup containing the support of~$\mu$. Furthermore, we assume that $\mu$ is contracting, i.e., there exists a sequence $\{M_j\}_{j}$
in~$T_\mu$ such that $\|M_j\|^{-1} M_j$ converges to a  rank-$1$ matrix. 

F\"urstenberg's theorem~\cite{Fur}, says that for $d=2$ the co-cycle generated by the random sequence $\{Y_j\}\subset SL(2,\R)$ has Lyapunov exponents $\la_1>0>\la_2=-\la_1$.
Moreover, there exists a unique $\mu$-invariant probability measure on~$P\R^1$, denoted by~$\nu$,  such that 
\[
\la_1 = \int_{P\R^1}\int_{SL(2,\R)} \log \| M\cdot \hat x\| \, \mu(dM)\, \nu(d\hat{x}) 
\]
where $M\cdot\hat{x}$ is the direction of $Mx$ where $x$ belongs to the equivalence class~$\hat{x}$.  
If these conditions are valid for all exterior powers $\Lambda^p Y_1$,  a theorem of Guivarc'h and Raugi, see~\cite[page 78]{BL}, extends F\"urstenberg's result to $d>2$ ensuring that all exponents are distinct. 

By a theorem of Le Page~\cite{Lepage}, we further know that one has exponential convergence to the invariant measure~$\nu$, see \cite[page 106]{BL}
for the precise meaning.  Finally, assuming that $\mu$ has exponential moments, one has a large deviation estimate of the following form: for every $\de>0$
\[
\limsup_{n\to\I} n^{-1} \log \Prob\Big(|\log \|Y_nY_{n-1}\cdots Y_2 Y_1\| - n\la_1|> n\de\Big) < -c(\de)<0
\]
see \cite[Theoreme 7]{Lepage}, or \cite[page 131]{BL}.  We therefore have all ingredients in order to apply the exact same argument as above.
Starting from the case where the random co-cycle does not depend on a parameter, we conclude for example that Corollary~\ref{cor:geometric} holds (with $E$ fixed).  Thus, one either has exponential convergence of the Lyapunov exponents, or~\eqref{1overn} holds. However, by the aforementioned convergence result
of Le Page it follows that~\eqref{1overn} does not occur in the random case and one always has exponential convergence. In particular, this logic shows that
the statement of
Corollary~\ref{cor:geometric} {\em is sharp}.

For random co-cycles depending ``nicely" on a parameter, we may again conclude that the exponents are H\"older in the parameter. 
However, we see no need to make this precise; indeed, if the dynamics is random or strongly mixing then we might expect much better regularity of the
exponents. At least for Schr\"odinger co-cycles this is indeed the case, see Campanino, Klein~\cite{CamKl} and Simon, Taylor~\cite{SiTa}. 
The approach we followed here is clearly not able to capture results of that strength; conversely, the techniques in these papers
 cannot handle deterministic dynamics such as shifts.

\vspace{1cm}

\centerline{\scshape W. Schlag}
\medskip
{\footnotesize
  \centerline{Department of Mathematics, The University of Chicago}
\centerline{5734 South University Avenue, Chicago, IL 60615, U.S.A.}
\centerline{\email{schlag@math.uchicago.edu}}
}

\end{document}